\documentclass[preprint]{elsarticle}

\usepackage{amssymb}
\usepackage{amsthm}
\usepackage{amsmath}

\usepackage{multirow}
\usepackage[usenames, dvipsnames]{color}

\newtheorem{theorem}{\bf Theorem}[section]

\newtheorem{definition}{\bf Definition}[section]
\newtheorem{lemma}{\bf Lemma}[section]

\newcommand{\E}{\mathrm{{E}}}

\newcommand{\Prob}{\mathrm{{P}}}

\newcommand{\dd}{\mathrm{d}}
\newcommand*{\eqb}{\begin{eqnarray}}
\newcommand*{\eqe}{\end{eqnarray}}

\newcommand{\fconv}{\Rightarrow}

\newcommand{\indyk}{{\mathbf{1}}}

\begin{document}

\begin{frontmatter}
\title{Scaling limits for L\'evy walks with rests}

\author{Marek Teuerle}
\address{Faculty of Pure and Applied Mathematics, Hugo Steinhaus Center,\\ Wroc{\l}aw University of Science and Technology,\\ Wybrze{\.z}e Wyspia{\'n}skiego 27, 50-370 Wroc{\l}aw, Poland}
\ead{marek.teuerle@pwr.edu.pl}


\begin{abstract}
	In this paper we investigate the asymptotic properties of the wait-first and jump-first L\'evy walk with rest, which is a generalization of standard jump-first and jump-first L\'evy walk that assumes each waiting time in the model is a sum of two positive random variables. We investigate the asymptotic properties of the theses new-type waiting times. Next we use the previous results of this paper together with continuous mapping approach to establish the main result, which is a functional convergence in Skorokhod $\mathbb{J}_1$ topology for the L\'evy walks with rests.
\end{abstract}
\end{frontmatter}
%



\section{Introduction}

One of the most fundamental results concerning the weak convergence of finite dimensional distributions of random walks by Donsker \cite{donsker1952} states that the properly normalized and shifted random walk with jumps of finite variance converges to the Brownian motion. 
$$n^{-1/2}\left(\sum^{\left\lfloor n t\right\rfloor}_{i=1}X_i-n\mathrm{E} X_i\right)\stackrel{f.d}{\rightarrow}B(t),\quad\mathrm{as}\quad n\to\infty.$$

In case of the continuous-time random walk the scaling limiting process depends strongly on the properties of the waiting times and jumps. For instance, if the the waiting times are heavy-tailed with index $\alpha\in (0;1)$ and independent of the jumps that are heavy-tailed with different index $\beta\in (\beta)$  the scaling limit is an $\alpha$-stable process subordinated to the inverse $\beta$-stable subordinator \cite{Meerschaert31}:
$$
a_n\left(\sum^{N(b_n t)}_{i=1}X_i\right)\stackrel{f.d}{\rightarrow}\left(L^{-}_\alpha(S_\beta^{-1}(t))\right)^+,\quad\mathrm{as}\quad n\to\infty,
$$
where $a_n$ and $b_n$ are appropriate sequences of constants, N(t) is the renewal process generated be the sequence of waiting times and the limit is the right-continuous version of a left-continuous version of $\alpha$-stable process subordinated to the inverse $\beta$-stable subordinator. 

In this paper we investigate a functional convergence in Skorokhod $\mathbb{J}_1$ topology, which generalizes the convergence of finite dimensional distributions, for the special case of coupled continuous-time random walk, called L\'evy walk with rests

This article is structured as follows. In \S 2 we recall the construction and main properties of various continuous-time random walks. Their special case with heavy-tailed waiting times equal (up to some multiplicative constant) to the length of jumps lead us to the class of L\'evy walks.  In \S 3, we introduce a new generalization of L\'evy walks with allows additional resting times before every jumps. The main results of this work are presented in \S 4. First, we present the asymptotic properties of the new waiting times with rests. Secondly, we establish the functional convergence in Skorokhod $\mathbb{J}_1$ topology of introduced class of L\'evy walks with rests. The last section is the summary of the obtained results and their possible applications.

\section{L\'evy walks}

Continuous-time random walk (CTRW), which is a generalization of classical random walk \cite{Pearson, Feller} was introduced to mathematics by Weiss and  Montroll in 1965 \cite{Montroll}.
The classical uncoupled CTRW is a \emph{c\'adl\'ag} stochastic process, in which every jump of random length occurs after random waiting time. Moreover, we assume that successive waiting times and jumps are independent from each other and from the previous pairs.  Now, CTRWs are well established mathematical models \cite{KlafterSokolov,Komorowski2005,MetzlerKlafter} with a rich class of physical applications: modelling of charge carrier transport in amorphous semiconductors \cite{Scher}, subsurface tracer dispersion \cite{Scher2}, electron transfer \cite{Nelson}, behavior of dynamical systems \cite{Bologna},  noise in plasma devices \cite{Chechkin2002}, self-diffusion in micelles systems \cite{Ott}, models in gene regulation \cite{Lomholt2005} to dispersion in turbulent systems \cite{Solomon1993}, to name only few.

The CTRW with heavy-tailed waiting times was the stochastic argument in the derivation of the fractional anomalous diffusion equations \cite{MetzlerKlafter}, where the fractional operator affects the time domain. On the other hand one can also consider a CTRW with heavy-tailed jumps and light-tailed waiting times (with finite variance), which is called L\'evy flight (LF). This last process has a non physical property, which is diverging variance. 

A class of processes that combines successfully two previously mentioned  elements, which are heavy-tailed features in space domain and finite variance of trajectory is called wait-first L\'evy walk \cite{Klafter2, Klafter1, Magdziarz2012, Teuerle2012, MagdziarzLW2015}. Mathematically it is simply a coupled CTRWs model with the waiting time equal (up to some multiplicative constant) to the length of the corresponding jump. Such a construction lead us to property that at any point of the trajectory its mean-square displacement is finite, although the length of the jumps are heavy-tailed.  We can also distinguish at least two other types of L\'evy walks. For the recent comprehensive review over latest results concerning L\'evy walks see \cite{LW_review,TeuerleOver2012}. In the following part we present formal definition of L\'evy walks.

\begin{definition}
Let $\{(T_i,\mathbf{J}_i)\}_{i\geq 1}$ be a sequence of independent and identically distributed (IID) random vectors such that $\Prob(T_i>0)=1$ and $T_i$ is heavy-tailed with index $\alpha\in(0;1)$, i.e.
\begin{equation}
\lim_{t\to \infty}t^\alpha\Prob(T_i>t)=c_1, \qquad\mathrm{where}\quad c_1>0,
\label{heavytailed}
\end{equation}
and the jumps $\mathbf{J}_i$ are equal to
\begin{equation}
\mathbf{J}_i=v_0\mathbf{V}_iT_i,\quad 
\label{jumps}
\end{equation}
where $\{\mathbf{V}_i\}_{i\geq 1}$ is an IID sequence of non-degenerate unit vectors from $\mathbb{R}^d$ independent of $\{T_i\}_{i\geq 1}$,

then, for the counting process 
$$N_t=\max\{n:T_1+T_2+\ldots+T_n\leq t\},$$
the processes
$$
\mathbf{U}(t)=\sum_{i=1}^{N_t}\mathbf{J}_i,\qquad
\mathbf{O}(t)=\sum_{i=1}^{N_t+1}\mathbf{J}_i,\qquad
$$
are called wait-first L\'evy walk, jump-first L\'evy walk, respectively.
\end{definition}

In next parts of this work we will use a term L\'evy walk for the family of above defined processes.

The LWs have a very intuitive behaviour and can be used to model random phenomena at the microscopic level. Namely, if parameter $v_0=1$ (this assumption is held thereafter for the simplicity), then a particle whose position is described by wait-first LW starts its motion at the origin and stays there for the random time $T_1$, then it makes the first jump $\mathbf{J}_1$. Next, it stays at the new location for the random time $T_2$ before making next jump $\mathbf{J}_2$, and then the scheme repeats. On the other hand, a particle whose position is described by jump-first LW starts its motion by performing $\mathbf{J}_1$ immediately after start and then waits in that location for random time $T_1$ and then the scheme repeats. In both cases due to relation (\ref{jumps}), the length of each jump $\mathbf{J}_i$ is equal to the length of the preceding waiting time ${T}_i$, while the direction of each jump $\mathbf{J}_i$ is governed by the random vector $\mathbf{V}_i$.  It is worth mentioning that both processes belongs to the class of \textit{c\'adl\'ag} processes. Another very important difference from the modelling point of view is that wait-first LW has all moments finite, while jump-first LW process has all moments infinite \cite{Magdziarz2015,Magdziarz2012,Teuerle2012,TeuerleOver2012, MagdziarzLW2015}.

L\'evy walks have been found as excellent models in the description of various phenomena, such as anomalous diffusion in complex systems \cite{MetzlerKlafter,Sokolov,Escande,MagdziarzWeron, Zumofen},
through human travel \cite{Brockmann2,Gonzales} and epidemic spreading \cite{Brockmann3, Dybiec}, to the foraging patterns of micro-organisms and animals \cite{Edwards,Buchanan, Sims, Viswanathan} as well as the migration of swarming bacteria \cite{bacLW2015}, to name only few (see examples of applications in \cite{LW_review} and references therein).  

Thus, it is a strong motivation for us to develop mathematical theory and further generalizations of L\'evy walks that include a additional resting time, esp. due to some recent work about the neuronal messenger ribonucleoprotein transport where L\'evy walks with this additional rests are seems to be a proper model \cite{newNature}. Obtained results, besides being interesting mathematically themselves, can be also applied in the statistical analysis of the above mentioned physical and biological systems. The numerical approximations developed on our results can be used by physicians and ecologist in verifying their hypotheses. 

\section{L\'evy walks with random rests}
  
The generalization of L\'evy walks considered here is based on the assumption that in every pair of waiting time and jump, the waiting times is a sum of two positive random variables that can be dependent. During this additional part in the waiting time, which can be called a \textit{rest} (or \textit{resting time}) the value of the process does not change. This new construction is illustrated in Fig. \ref{trajektorie}. The formal definition of L\'evy walks with rests is given on the following part.

\begin{definition}
\label{def:lwwithrests}
Let $\{(R_i,T_i,\mathbf{J}_i)\}_{i\geq 1}$ be a sequence IID random vectors such that the waiting times $T_i$'s are heavy-tailed with index $\beta\in(0;1)$ and $\Prob(T_i>0)=1$, see (\ref{heavytailed}), \\ the jumps $\mathbf{J}_i$'s are constructed as follows:
$$
\mathbf{J}_i=\mathbf{V}_i T_i,
$$
and
the resting times $R_i$'s are heavy-tailed with index $\beta\in(0;1)$, i.e.
$$
\lim_{t\to \infty}t^\beta\Prob(R_i>t)=c_2, \qquad\mathrm{where}\quad c_2>0,
$$

$\Prob(R_i>0)=1$ and are independent of any $T_i$'s and $\mathbf{J}_i$'s,
then, for the counting process 
$$N^R_t=\max\{n:T_1+R_1+T_2+R_2+\ldots+T_n+R_n\leq t\},$$
the processes
\begin{equation}
\mathbf{U}^R(t)=\sum_{i=1}^{N^R_t}\mathbf{J}_i,\qquad
\mathbf{O}^R(t)=\sum_{i=1}^{N^R_t+1}\mathbf{J}_i,\qquad
\label{lwwithrests}
\end{equation}
are called wait-first L\'evy walk with rests and jump-first L\'evy walk with rest, respectively.
\end{definition}

\begin{figure}[!h]
\centering\includegraphics[scale=0.75]{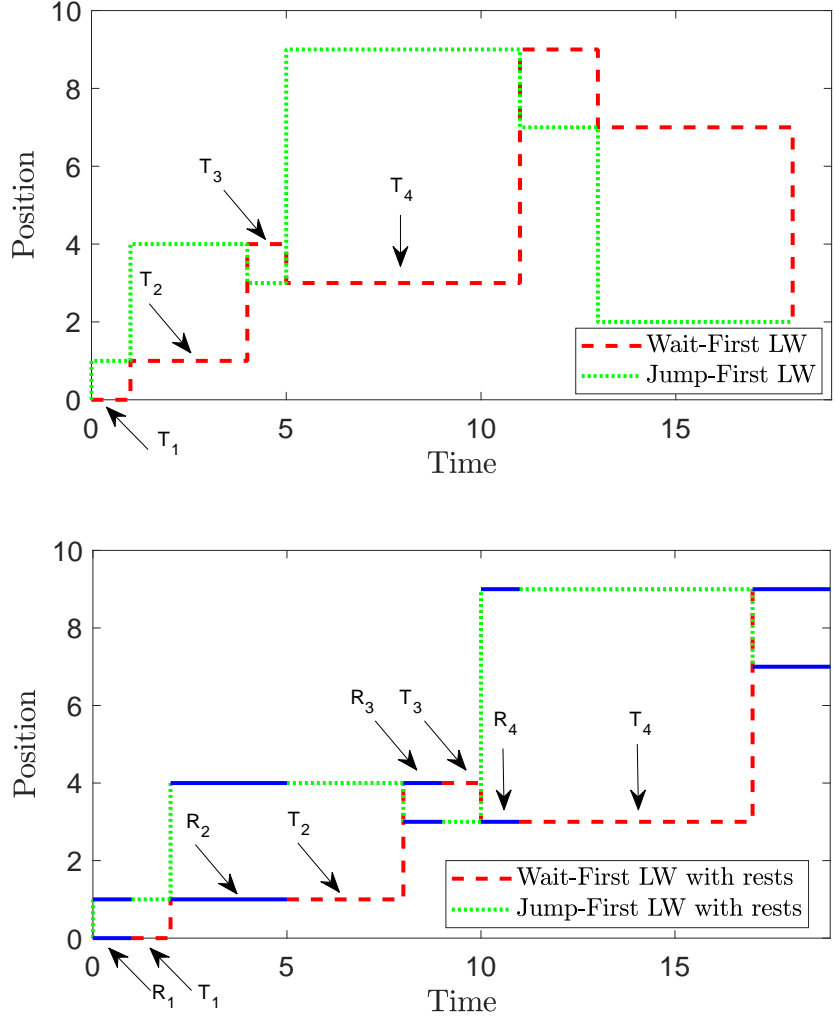}
\caption{Exemplary trajectories of wait-first and jump-first LW (top panel) and wait-first and jump-first LW with rests (bottom panel) generated by the same sequence of $\{T_i, \mathbf{J}_i\}^{6}_{i=1}$. Please note that the resting times $R_i$'s (blue solid line) proceed respective waiting times $T_i$'s.}
\label{trajektorie}
\end{figure}
The modified counting process $N^R_t$ in all considered processes defined in Eq. (\ref{lwwithrests}) lead us to the new class of coupled CTRWs walks. In next parts of this work we will use a term L\'evy walk with rests for the family of processes defined in (\ref{lwwithrests}).

Asymptotic analysis of these constructions is a first step towards the asymptotic analysis of the family of continuous coupled CTRWs with random rests described in \cite{KlafterSokolov,ZaburdaevChukbar,LW_review}.

\section{Functional convergence for the wait-first and jump-first L\'evy walks with rest}

In our further consideration we use the general framework of L\'evy process therefore we recall here its definition \cite{sato}.
We recognize process $\mathbf{X}(t)$ as a $d$-dimensional L\'evy process if its Fourier transform has the following form:
$$
\E \exp\{i\langle\mathbf{k},\mathbf{X}(t)\rangle\}=\exp\{-t\psi(\mathbf{k})\},\qquad \mathbf{k}=(k_1,k_2,\ldots,k_d)^T\in\mathbb{R}^d
$$
where the function $\psi(\mathbf{k})$ is the following Fourier symbol:
\begin{equation}
\label{levy-khinchin_exponent}
\psi(\mathbf{k})=-i\langle\mathbf{k},\mathbf{a}\rangle+\frac{1}{2}\langle\mathbf{k},\mathbf{Qk}\rangle+\int\limits_{\mathbf{x}\neq\mathbf{0}}{\left(1+\frac{i\langle\mathbf{k},\mathbf{x}\rangle}{1+\|\mathbf{x}\|^2}-\exp(i\langle\mathbf{k},\mathbf{x}\rangle)\right)\nu(\dd x)}.
\end{equation}
In Eq. (\ref{levy-khinchin_exponent}) the symbol $\langle\cdot,\cdot\rangle$ denotes standard inner product, $\|\mathbf{\cdot} \|$ is the length of given vector, $\mathbf{a}\in\mathbb{R}^d$ is the drift, $\mathbf{Q}$ is the Gaussian covariance matrix, \ $\nu$ is the L\'evy measure $\mathbb{R}^d\setminus\{\mathbf{0}\}$ satisfying $\int\limits_{\mathbf{x}\neq\mathbf{0}}\min(\|\mathbf{x}\|^2,1)\nu(\dd x)<\infty$. The triplet $[\mathbf{a},\mathbf{Q},\mathbf{\nu}]$ is called L\'evy triplet and uniquely determines L\'evy process $\mathbf{X}(t)$. In the next part we will characterize different L\'evy processes, that arise as functional limits if considered models, using their the L\'evy triplet.

\subsection{Asymptotic properties of the waiting times with rests}
Before we present our main result, we deal with the asymptotic properties of the modified waiting times. 

Let us recall that the waiting times are heavy-tailed distributed with $\alpha\in (0;1)$, (see \ref{heavytailed}), therefore the following convergence holds \cite{Meerschaert1}:
\begin{equation}
n^{-1/\alpha}\sum_{i=1}^{[nt]}{T_i}\fconv S^{}_\alpha(t),\quad \mathrm{as}\ n\to\infty,
\label{waitingtimesfconv}
\end{equation}
where the scaling limit $S^{}_\alpha(t)$ is called an $\alpha$-stable subordinator fully characterized by its L\'evy triplet
\begin{equation}
\label{waitingtimeslevytriplet}
\left[\ \int\limits_{x\neq 0}\frac{x}{1+x^2} \nu(\dd x)
,0,\nu(\dd x)=\frac{\alpha  x^{-\alpha-1}}{\Gamma(1-\alpha)}\mathbf{1}_{\{x>0\}} \dd x\right],
\end{equation}
where $x\in \mathbb{R}_+$.
Symbol "$\Rightarrow$" used in formula (\ref{waitingtimesfconv}) denotes weak functional convergence in the Skorokhod $\mathbb{J}_1$ topology, for details see \cite{Billingsley,Whitt}, which also implies the convergence in distribution.

In the next lemma we derive the scaling limit for the partial sums of modified waiting times.

\begin{lemma}
\label{lemma:lwwithrestsfconv}
Let $\left\{(T_i,{R}_i)\right\}_{i\geq 1}$ be the sequence of iid waiting times and rests as in Definition \ref{def:lwwithrests}. Then:
\begin{itemize}
\item[(i)] if $T_i$ and ${R}_i$ are independent for all $i\geq 1$, then:
\begin{equation}
\label{waitingrestsfconv1}
n^{-1/\min(\alpha,\beta)} \sum_{i=1}^{[nt]}\left( T_i+R_i\right)\fconv S_{\min(\alpha,\beta)}(t),\quad \mathrm{as}\ n\to\infty,
\end{equation}
where $S_{\min(\alpha,\beta)}(t)$ is the L\'evy process with the following L\'evy triplet
$$
\left[\ \int\limits_{x\neq 0}\frac{x}{1+x^2} \nu(\dd x)
,0,\nu(\dd x)=\frac{\min(\alpha,\beta) x^{-\min(\alpha,\beta)-1}}{\Gamma(1-\min(\alpha,\beta))}\mathbf{1}_{\{x>0\}} \dd x\right],
$$
where $x\in \mathbb{R}_+$,
\item[(ii)] if $T_i={R}_i$ for all $i\geq 1$, then:
\begin{equation}
\label{waitingrestsfconv2}
n^{-1/{\alpha}} \sum_{i=1}^{[nt]}\left( T_i+R_i\right)\fconv S^R_{\alpha}(t),\quad \mathrm{as}\ n\to\infty,
\end{equation}
where $S^R_{\alpha}(t)$ is the L\'evy process with the following L\'evy triplet
\begin{equation}
\label{waitingreststriplet2}
\left[\ \int\limits_{x\neq 0}\frac{x}{1+x^2} \nu_R(\dd x)
,0,\nu_R(\dd x)=\frac{\alpha 2^\alpha x^{-\alpha-1}}{\Gamma(1-\alpha)}\mathbf{1}_{\{x>0\}} \dd x\right],
\end{equation}
where $x\in \mathbb{R}_+$.
\end{itemize}

\end{lemma}
\begin{proof}
The proof of Lemma \ref{lemma:lwwithrestsfconv} is presented in the Appendix A.
\end{proof}

Asymptotically, as shown in \cite{Magdziarz2015}, the $d$-dimensional jumps of the form $\mathbf{J}_i=\mathbf{V_i} T_i$ converge functionally to the $d$-dimensional $\alpha$-stable process, that are fully described by the asymptotics of the waiting times $T_i$ and the distribution of $V_i$, namely:
\begin{align}
\label{jumpfconv}
n^{-1/\alpha} \sum_{i=1}^{[nt]} \mathbf{J}_i\fconv \mathbf{L}_{\alpha}(t),\quad \mathrm{as}\ n\to\infty,
\end{align}
where $\mathbf{L}_{\alpha}(t)$ is described by the following L\'evy triplet
$$
\left[\ \int\limits_{\mathbf{x}\neq \mathbf{0}}\frac{\mathbf{x}}{1+\|\mathbf{x}\|^2}\mathbf{\nu}_{\mathbf{L}^{}_{\alpha}}(\dd \mathbf{x}),\mathbf{0},\mathbf{\nu}_{\mathbf{L}^{}_{\alpha}}(\dd \mathbf{x})=\int\limits_{\mathbf{D}}\frac{\alpha \|\mathbf{x}\|^{-\alpha-1}}{\Gamma(1-\alpha)}\mathbf{1}_{\{\|\mathbf{x}\|>0\}} \mathbf{\Lambda}(\dd \mathbf{u})\right],
$$
where $\mathbf{x}\in \mathbb{R}^d\setminus\{\mathbf{0}\}$, $\mathbf{u}=\mathbf{x}/\|\mathbf{x}\|\in \mathbf{D}\subset \mathbb{S}^{d-1}$, $\mathbf{\nu}_{\mathbf{L}^{}_{\alpha}}$ is the L\'evy measure given by Eq.~(\ref{waitingtimeslevytriplet}) and $\Lambda(\dd \mathbf{u})=\Prob(\mathbf{V}_1\in \dd \mathbf{u})$ is called the spectral measure.

\subsection{Asymptotic properties of joint waiting times with rests and jumps}

In the next lemma we derive the scaling limit of joint waiting times with rests and jumps partial sums, which be a crucial step in the proof of the main result.
\begin{lemma}
\label{lemma:joinfconv}
Let $\{(R_i,T_i,\mathbf{J}_i)\}_{i\geq 1}$ be a sequence of random vectors as in Definition \ref{def:lwwithrests}. Then:
as $n\to\infty$:
\begin{itemize}
\item[(i)] if $T_i$ and ${R}_i$ are independent for all $i\geq 1$ and the index of stability $\alpha$ of $T_i$ is smaller then the index of stability $\beta$ of $R_i$, the following holds:
\begin{equation}
\label{jointfconv1}
\left(n^{-1/\alpha} \sum_{i=1}^{[nt]} \mathbf{J}_i,n^{-1/\alpha} \sum_{i=1}^{[nt]} (T_i+R_i) \right)\fconv \left(\mathbf{L}_{\alpha}(t),S_{\alpha}(t)\right),
\end{equation}
where $\left(\mathbf{L}_{\alpha}(t),S_{\alpha}(t)\right)$ is the $(d+1)$-dimensional L\'evy process described by the L\'evy triplet
\begin{equation}
\label{jointlevytriplet1}
\left[\ \int\limits_{\mathbf{x}\neq \mathbf{0}}\frac{\mathbf{x}}{1+\|\mathbf{x}\|^2}\mathbf{\nu}_{(\mathbf{L}^{}_{\alpha},S^{}_\alpha)}(\dd \mathbf{x}),\mathbf{0},\mathbf{\nu}_{(\mathbf{L}^{}_{\alpha},S_\alpha)}(\dd \mathbf{x})=\int\limits_{\mathbf{D}}\delta_{(\mathbf{u} t)}(\dd \mathbf{x}_1)  \nu(\dd t)\mathbf{\Lambda}(\dd \mathbf{u})\right],
\end{equation}
\item[(ii)] if $T_i$ and ${R}_i$ are independent for all $i\geq 1$ and the index of stability $\alpha$ of $T_i$ is greater then the index of stability $\beta$ of $R_i$, the following holds:
\begin{equation}
\label{jointfconv2}
\left(n^{-1/\alpha} \sum_{i=1}^{[nt]} \mathbf{J}_i,n^{-1/\beta} \sum_{i=1}^{[nt]} (T_i+R_i) \right)\fconv \left(\mathbf{L}_{\alpha}(t),S_{\beta}(t)\right),
\end{equation}
where $\left(\mathbf{L}_{\alpha}(t),S_{\beta}(t)\right)$ is the $(d+1)$-dimensional L\'evy process described by the L\'evy triplet
\begin{equation}
\label{jointlevytriplet2}
{\left[\ \! \int\limits_{\mathbf{x}\neq \mathbf{0}}\!\!\!\frac{\mathbf{x}}{1+\|\mathbf{x}\|^2}\mathbf{\nu}_{(\mathbf{L}^{}_{\alpha},S^{}_\beta)}(\dd \mathbf{x}),\mathbf{0},\mathbf{\nu}_{(\mathbf{L}^{}_{\alpha},S_\beta)}(\dd \mathbf{x})\!=\!\!\int\limits_{\mathbf{D}}\!\left(\nu(\dd (\mathbf{u} t)) \delta_0(\dd t)+ \nu(\dd t)\delta_\mathbf{0}(\dd (\mathbf{u} t))\right)\mathbf{\Lambda}(\dd \mathbf{u})\right],}
\end{equation}

\item[(iii)] if $T_i$ and ${R}_i$ are dependent according to the relation $T_i={R}_i$ for all $i\geq 1$, the following holds:
\begin{equation}
\label{jointfconv3}
\left(n^{-1/\alpha} \sum_{i=1}^{[nt]} \mathbf{J}_i,n^{-1/\alpha} \sum_{i=1}^{[nt]} (T_i+R_i) \right)\fconv \left(\mathbf{L}_{\alpha}(t),S^R_{\alpha}(t)\right),
\end{equation}
where $\left(\mathbf{L}_{\alpha}(t),S_{\alpha}(t)\right)$ is the $(d+1)$-dimensional L\'evy process described by the L\'evy triplet
\begin{equation}
\label{jointlevytriplet3}
\left[\ \int\limits_{\mathbf{x}\neq \mathbf{0}}\frac{\mathbf{x}}{1+\|\mathbf{x}\|^2}\mathbf{\nu}_{(\mathbf{L}^{}_{\alpha},S^{R}_\alpha)}(\dd \mathbf{x}),\mathbf{0},\mathbf{\nu}_{(\mathbf{L}^{}_{\alpha},S^{R}_\alpha)}(\dd \mathbf{x})=\int\limits_{\mathbf{D}}\delta_{(\mathbf{u} t)}(\dd \mathbf{x}_1)  \nu_R(\dd t)\mathbf{\Lambda}(\dd \mathbf{u})\right],
\end{equation}
\end{itemize}
where $\mathbf{x}=(\mathbf{x}_1,t)\in \mathbb{R}^d\setminus\{\mathbf{0}\}\times\mathbb{R}_+$,  $\mathbf{u}=\mathbf{x_1}/\|\mathbf{x_1}\|\in \mathbf{D}\subset \mathbb{S}^{d-1}$, $\mathbf{\nu}_R$ is the L\'evy measure given by Eq.~(\ref{waitingreststriplet2}), $\nu_R(\dd t)=2^\alpha \nu (\dd t)$ and $\mathbf{\Lambda}(\dd \mathbf{u})=\Prob(\mathbf{V}_1\in \dd \mathbf{u})$.
\end{lemma}
\begin{proof}
The proof of Lemma \ref{lemma:joinfconv} is presented in the Appendix B.
\end{proof}

\subsection{Main result}
Before we state the main result of this work, we need to introduce a left-continuous version of any right-continuous processes $\mathbf{X}(t)$ by putting:
$
\label{X_left_limit}
\mathbf{X}^-(t) = \lim_{s \nearrow t} \ \mathbf{X}(s),
$
and the right-continuous version of any left-continuous process $\mathbf{Y}(t)$ by putting:
$
\label{Y_right_limit}
\mathbf{X}^+(t) = \lim_{s \searrow t} \ \mathbf{Y}(s).
$
 As it will be shown in the theorem below, $\mathbf{L}_{\alpha}^-(t)$ appears as a driving process in the continuous-time limit of the wait-first LW. The occurrence of the process $\mathbf{L}_{\alpha}^-(t)$ in the limit of wait-first LWs is the consequence of the fact that the walker follows the so called 'wait-jump' scenario, i.e. $i$-th jump is preceded by $i$-th waiting time. In the case of jump-first LW the waiting time of the walker is preceded by the corresponding jump, making the walker following 'jump-wait' scenario. The second scheme leads to infinite moments of the jump-first LW scaling limit \cite{TeuerleOver2012}. Moreover, we introduce the inverse $\alpha$-stable subordinator $S_\alpha^{-1}(t)$, which is defined as \cite{Meerschaert31}
$
S^{-1}_{\alpha}(t)=\inf\{\tau \geq 0: S_{\alpha}(\tau)>t\}.
$
The process $S^{-1}_{\alpha}(t)$ is the scaling limit of the counting process $N^R(t)$ generated by the sequence
of waiting times $\{T_i+R_i\}_{i\geq 1}$ in case when $\alpha<\beta$.  It appears in the next theorem  that the processes $S^{-1}_{\alpha}(t), S^{-1}_{\beta}(t), (S^R_{\alpha})^{-1}(t)$ play a role of the internal operational time in the limits the L\'evy walks with rests.

\begin{theorem}
\label{th:fconv}
Let $\mathbf{U}^R(t)$ and $\mathbf{O}^R(t)$ be the wait-first and jump-first L\'evy walks with rests, respectively, that are generated by the sequence $\left\{(R_i,T_i,\mathbf{J}_i)\right\}_{i\geq 1}$ of waiting times, rests and jumps given in Definition \ref{def:lwwithrests}. Then:
\begin{itemize}
\item[(i)] if $T_i$ and ${R}_i$ are independent for all $i\geq 1$ and the index of stability $\alpha$ of $T_i$ is smaller then the index of stability $\beta$ of $R_i$, the following holds:
$$
n^{-1/\alpha}{\mathbf{U}^R\left(n^{1/\alpha} t\right)}\Rightarrow \left( \mathbf{L}^{-}_{\alpha}\!\!\left( S^{-1}_{\alpha}(t)\right)\right)^+,\\\nonumber n^{-1/\alpha}{\mathbf{O}^R\left(n^{1/\alpha} t\right)}\Rightarrow  \mathbf{L}^{}_{\alpha}\!\!\left( S^{-1}_{\alpha}(t)\right), \quad \mathrm{as}\ n\to\infty,
$$
where the processes $\mathbf{L^{}}_{\alpha}(t)$ and $S^{}_{\alpha}(t)$ are dependent according to their joint L\'evy measure given by Eq. (\ref{jointlevytriplet1}),
\item[(ii)] if $T_i$ and ${R}_i$ are independent for all $i\geq 1$ and the index of stability $\alpha$ of $T_i$ is greater then the index of stability $\beta$ of $R_i$, the following holds:
$$
n^{-1/\alpha}{\mathbf{U}^R\left(n^{1/\beta} t\right)}\Rightarrow \left( \mathbf{L}^{-}_{\alpha}\!\!\left( S^{-1}_{\beta}(t)\right)\right)^+,\\\nonumber n^{-1/\alpha}{\mathbf{O}^R\left(n^{1/\beta} t\right)}\Rightarrow  \mathbf{L}^{}_{\alpha}\!\!\left( S^{-1}_{\beta}(t)\right),\quad \mathrm{as}\ n\to\infty,
$$
where the processes $\mathbf{L^{}}_{\alpha}(t)$ and $S^{}_{\alpha}(t)$ are independent according to their joint L\'evy measure given by Eq. (\ref{jointlevytriplet2}),
\item[(iii)] if $T_i$ and ${R}_i$ are dependent according to the relation $T_i={R}_i$ for all $i\geq 1$, the following holds:
$$
n^{-1/\alpha}{\mathbf{U}^R\left(n^{1/\alpha} t\right)}\Rightarrow \left( \mathbf{L}^{-}_{\alpha}\!\!\left( \left(S^R\right)^{-1}_{\alpha}(t)\right)\right)^+,\\\nonumber n^{-1/\alpha}{\mathbf{O}^R\left(n^{1/\alpha} t\right)}\Rightarrow  \mathbf{L}^{}_{\alpha}\!\!\left( \left(S^R\right)^{-1}_{\alpha}(t)\right),
$$
as $n\to\infty$, where the processes $\mathbf{L^{}}_{\alpha}(t)$ and $S^{R}_{\alpha}(t)$ are dependent according to their joint L\'evy measure given by Eq. (\ref{jointlevytriplet3}).
\end{itemize}
\end{theorem}
\begin{proof}
The proof of Theorem \ref{th:fconv} is presented in the Appendix C.
\end{proof}

%

%

Finally, it appears that the scaling limits of L\'evy walks with rest are the $\alpha$-stable processes subordinated to the inverse $\alpha$-stable subordinators. We distinguish two main scenarios: first one (case (i) and (iii)) in which the limiting processes is the $\alpha$-stable processes and is subordinated to strongly dependent $\alpha$-stable subordinator (the dependence is precisely defined by their L\'evy triplets, see (\ref{jointlevytriplet3}) and (\ref{jointlevytriplet1})) and second one (case (ii)), in which limiting process is the $\alpha$-stable process independent to its $\beta$-stable inverse subordinator.  

\section{Conclusion}
In this paper we propose a modification of coupled CTRW process, that assume that every waiting time of the model is a sum of two independent positive random variables with at least one of them strongly coupled with the respective jump. The introduced model, depending on the fact that jump are preceding or following the modified waiting time, is called wait-first and jump-fist L\'evy walks with rest, respectively. The above modification is important from the modelling point of view, it allows to include additional waiting time in the trajectories that might be coupled or uncoupled with the jumps. The main result of this paper is asymptotic properties of L\'evy walks with rests by means of the functional convergence in Skorokhod $\mathbb{J}_1$ topology. Our finding also categorize the form of the limiting process depending on the dependence structures between original waiting times and resting times. It is worth noticing that the limiting process is in general of the same form, namely it a $\alpha$-stable process subordinated by another inverse $\alpha$-stable or $\beta$-stable process (in accordance to our notation). The difference is that in the first case the external processes is dependent on the inverse $\alpha$-stable subordinator, while in the second case - independent on the inverse $\beta$-stable subordinator.

Let us note that for the standard wait-first and jump-first LW models were used to establish the functional convergence of the fully continuous L\'evy walk \cite{MagdziarzLW2015}. In analogy to this fact, our findings can be used further to obtain the functional convergence of the fully continuous L\'evy walk with rests \cite{LW_review}. We believe that outcomes of this paper will have also practical meaning for the analysis of large variety of phenomena observed in nature, since the limiting processes can be simulated and estimated by the known numerical procedures \cite{Magdziarz2015}.

\textbf{Acknowledgements.} {This work was supported by NCN, Maestro grant no. 2012/06/A/ST1/00258.}

\section*{Appendix}
\section*{Appendix A. Proof of Lemma \ref{lemma:lwwithrestsfconv}}

The proof of both parts of Lemma \ref{lemma:lwwithrestsfconv} follows from the Generalized Central Limit Theorem (GCLT). In case (i) we consider two independent random variables $T_i$ and $R_i$ that are in the domain of attraction of two stable random variables with indices of stability $\alpha$ and $\beta$, respectively. First let us make first observation, which is based on GCLT, that the sum of these random variables $T_i+R_i$ is on the domain of attraction of a stable distribution with the index of stability that is equal to the minimum of $\alpha$ and $\beta$ \cite{GnedenkoKolmogorov}. Therefore, upon the invariance principle for stable distributions and stable processes \cite{Meerschaert31,Stone}, we conclude that the functional convergence (or functional GCTL) given by (\ref{waitingrestsfconv1}) in Skorokhod $\mathbb{J}_1$ topology holds. The form of the L\'evy measure $\nu$ is a consequence of above facts and the form of analogous L\'evy measure in (\ref{waitingtimeslevytriplet}).

For the proof of part (ii) let us observe that a very strong assumption about the dependence between waiting times and rests, that is $T_i=R_i$, lead us to the observation that in fact we analyze a GCLT for the random variables that belongs the the domain of attraction of stable distribution in general the same (up to the scale parameter) as in $T_i$ case, see (\ref{waitingtimesfconv}). Similarly as in part (i), the invariance principle asserts the conclusion (\ref{waitingrestsfconv2}). Last we comment on the scale parameter of $S^R_{\alpha}(t)$.  The convergence (\ref{waitingtimesfconv}) implies the following behaviour of $T_i$'s: $$n\Prob(n^{-\frac{1}{\alpha}}T_1>x)\stackrel{n\to\infty}{\rightarrow}\nu((x,\infty))=\frac{ x^{-\alpha}}{\Gamma(1-\alpha)}\mathbf{1}_{\{x>0\}}.$$ Using the last property, we obtain that $$n\Prob(n^{-\frac{1}{\alpha}}(T_1+R_1)>x)=n\Prob(n^{-\frac{1}{\alpha}}(2T_1)>x)\stackrel{n\to\infty}{\rightarrow}\nu_R((x,\infty))=\frac{ 2^\alpha x^{-\alpha}}{\Gamma(1-\alpha)}\mathbf{1}_{\{x>0\}},$$
which by the uniqueness of L\'evy triplet asserts the form of L\'evy measure in (\ref{waitingreststriplet2}).

\section*{Appendix B. Proof of Lemma \ref{lemma:joinfconv}}

Before proving the main results of the lemma, let us recall the asymptotic properties of the jumps in considered model. As their construction is the following $\mathbf{J}_i=\mathbf{V}_iT_i$ (defined by (\ref{jumps})), then based on the result presented in \cite{Magdziarz2015}, we have the following convergence in the Skorokhod $\mathbb{J}_1$ topology:
$$
{n^{-1/\alpha}}\sum_{i=1}^{[nt]}\mathbf{J}_{i} \fconv  \mathbf{L}_{\alpha}(t),
$$
as $n\to\infty$, and $\mathbf{L}_{\alpha}(t)$ is a d-dimensional $\alpha$-stable process with the following L\'evy triplet:
$$
\left[\ \int\limits_{\mathbf{x}\neq\mathbf{0}} \frac{\mathbf{x}}{1+\|\mathbf{x}\|^2}\nu(\dd \mathbf{x}),\mathbf{0},\mathbf{\nu}(\dd \mathbf{x})=\frac{\alpha r^{-\alpha-1}\mathbf{\Lambda}(\dd \mathbf{u})}{\Gamma(1-\alpha)}\right]
$$
for $\mathbf{x}\in\mathbb{R}^d, r=\|\mathbf{x}\|,\mathbf{u}=\mathbf{x}/\|\mathbf{x}\|$ and  $\mathbf{\Lambda}(\dd \mathbf{u})=\Prob(\mathbf{V}_1\in\dd \mathbf{u})$,

Now, we prove the part (i) of Lemma \ref{lemma:joinfconv}. For any Borel sets $\mathbf{B}_1\in\mathcal{B}(\mathbb{R}^d)$ and $B_2 \in \mathcal{B}(\mathbb{R}_+ )$ such that  $\mathbf{B}_1(R,\mathbf{D})=\{r\mathbf{u}\in\mathbb{R}^d:r\in R, \mathbf{u} \in \mathbf{D}\}$, where $ R\in\mathcal{B}(\mathbb{R}_+)$ and $\mathbf{D}\in\mathcal{B}(\mathbb{S}^{d-1})$,  we have
\begin{align}
\label{proof:indepjoin}
& n\Prob(n^{-1/\alpha}{T}_1\mathbf{V}_1\in \mathbf{B}_1, n^{-1/\alpha}(T_1+R_1) \in B_2)=\\
&=n\! \int_{\mathbf{D}}\!\!\!\! \Prob(n^{-1/\alpha} T_1\mathbf{u} \in \mathbf{B}_1,n^{-1/\alpha}(T_1+R_1)\in B_2)\Prob(\mathbf{V}_1 \in \dd \mathbf{u})=\nonumber\\
&=n\! \int_{\mathbf{D}}\!\int_{B_2}\!\!\!\!\!\! \Prob( t\mathbf{u} \in \mathbf{B}_1)\Prob(\mathbf{V}_1 \in \dd \mathbf{u})\Prob(n^{-1/\alpha} (T_1+R_1) \in \dd {t})=\nonumber\\\nonumber
&= \int_{\mathbf{D}}\!\int_{B_2}\!\!\!\!\! \indyk(t\mathbf{u} \in \mathbf{B}_1)n\Prob(\mathbf{V}_1 \in \dd \mathbf{u})\Prob(n^{-1/\alpha} (T_1+R_1) \in \dd {t}),
\end{align}
which implies that for $(\mathbf{x},t)\in\mathbb{R}^d\setminus\{\mathbf{0}\}\times\mathbf{R}_+$ and $\mathbf{u}\in\mathbb{S}^{d-1}$ the following equality holds
\begin{align*}
n\Prob(n^{-1/\alpha}{T}_1\mathbf{V}_1\! \in \dd \mathbf{x},n^{-1/\alpha}(T_1+R_1) \in \dd t)\! =\int_{\mathbf{D}}\delta_{( t\mathbf{u})}(\dd  \mathbf{x})\Prob(\mathbf{V}_1 \in \dd \mathbf{u})n\Prob(n^{-1/\alpha} (T_1+R_1)\! \in \dd {t}),
\end{align*}
with the symbol $\delta_\mathbf{x}$ denoting the Dirac's delta at point $\mathbf{x}$.

By letting $n\to\infty$, we obtain from (\ref{proof:indepjoin}) and Lemma \ref{lemma:lwwithrestsfconv}, which deals with the asymptotics of waiting times with rests, that
\begin{align*}
&n\Prob(n^{-1/\alpha} \mathbf{J}_1\! \in \! \dd \mathbf{x}, n^{-1/\alpha} (T_{1}+R_1)\! \in\! \dd t)=\int_{\mathbf{D}}\delta_{(t\mathbf{u})}(\dd  \mathbf{x})n\Prob(n^{-1/\alpha} (T_1+R_1)\! \in \dd {t})\Prob(\mathbf{V}_1 \in \dd \mathbf{u})\\\nonumber
&\longrightarrow \int_{\mathbf{D}}\delta_{t\mathbf{u})}(\dd \mathbf{x})\nu(\dd t)\mathbf{\Lambda}(\dd \mathbf{u}).
\end{align*}
The last results leads to the fact that the limiting joint measure of $(\mathbf{L}_{\alpha}(1),S_{\alpha}(1))$ denoted as $\nu_{(\mathbf{L}_{\alpha},S_{\alpha})}$ is the following:
\begin{align}
\label{proof:indepjointlevymeasure}
\nu_{(\mathbf{L}_{\alpha},S_{\alpha})}(\dd \mathbf{x},\dd t)\stackrel{\mathrm{def}}{=} \int_{\mathbf{D}}\delta_{(t\mathbf{u})}(\dd \mathbf{x})\nu(\dd t)\mathbf{\Lambda}(\dd \mathbf{u})
\end{align}
Therefore condition (a) of Th. 3.2.2 \cite{Meerschaert1} is fulfilled.
Moreover, since the L\'evy measure in (\ref{proof:indepjointlevymeasure}) is a L\'evy measure of $d+1$-dimensional multidimensional $\alpha$-stable variable then the condition (b) of Th. 3.2.2 \cite{Meerschaert1} is fulfilled with $Q_{(\mathbf{L}_{\alpha},S_{\alpha})}\equiv\mathbf{0}$, which completes the proof of convergence in distribution. The last property together with Th. 4.1 \cite{Meerschaert1} (invariance principle) implies the functional convergence (\ref{jointfconv1}).

For the proof the part (ii) of Lemma \ref{lemma:joinfconv} observe the the limiting processes for jumps and waiting times with rests, which are $\mathbf{L}_\alpha(t)$ and $S_\beta(t)$ respectively,  are asymptotically independent since in this part $R_i$'s and $T_i$ are assumed to be independent, see (\ref{jumpfconv}) and (\ref{waitingrestsfconv2}). Therefore for the functional convergence of the joint process of partial sums we get a joint process $(\mathbf{L}_\alpha(t), S_\beta(t))$ that consists of independent parts. By the Corollary 2.3 in \cite{Meerschaert31} we find the L\'evy triplet of the limiting process $(\mathbf{L}_\alpha(t), S_\beta(t))$.

To complete this proof, we shortly comment on the part (iii). Let us recall that in this case the rest $R_i$ and waiting times $T_i$ are strongly coupled by the relation $T_i=R_i$, therefore analogously to the part (i), we have:

\begin{align}
\label{proof:depjoin}
& n\Prob(n^{-1/\alpha}{T}_1\mathbf{V}_1\in \mathbf{B}_1, 2n^{-1/\alpha}T_1 \in B_2)=\\\nonumber
&= \int_{\mathbf{D}}\!\int_{B_2}\!\!\!\!\! \indyk(t\mathbf{u} \in \mathbf{B}_1)n\Prob(\mathbf{V}_1 \in \dd \mathbf{u})\Prob(2n^{-1/\alpha} T_1 \in \dd {t}),
\end{align}
which implies that for $(\mathbf{x},t)\in\mathbb{R}^d\setminus\{\mathbf{0}\}\times\mathbf{R}_+$ and $\mathbf{u}\in\mathbb{S}^{d-1}$ the following equality holds
\begin{align*}
n\Prob(n^{-1/\alpha}{T}_1\mathbf{V}_1\! \in \dd \mathbf{x},2n^{-1/\alpha}T_1 \in \dd t)\! =\int_{\mathbf{D}}\delta_{( t\mathbf{u})}(\dd  \mathbf{x})\Prob(\mathbf{V}_1 \in \dd \mathbf{u})n\Prob((2^{-\alpha} n)^{-1/\alpha} T_1\! \in \dd {t}).
\end{align*}

Again, by letting $n\to\infty$, we obtain from (\ref{proof:depjoin}) and Lemma \ref{lemma:lwwithrestsfconv} that
\begin{align*}
&n\Prob(n^{-1/\alpha} \mathbf{J}_1\! \in \! \dd \mathbf{x}, n^{-1/\alpha} (T_{1}+R_1)\! \in\! \dd t)=\int_{\mathbf{D}}\delta_{(t\mathbf{u})}(\dd  \mathbf{x})n\Prob((2^{-\alpha} n)^{-1/\alpha} T_1\! \in \dd {t})\Prob(\mathbf{V}_1 \in \dd \mathbf{u})\\\nonumber
&\longrightarrow \int_{\mathbf{D}}\delta_{t\mathbf{u})}(\dd \mathbf{x})2^\alpha\nu(\dd t)\mathbf{\Lambda}(\dd \mathbf{u}).
\end{align*}
The last results leads to the fact that the limiting joint measure of $(\mathbf{L}_{\alpha}(1),S^R_{\alpha}(1))$ denoted as $\nu_{(\mathbf{L}_{\alpha},S^R_{\alpha})}$ is the following:
$$
\nu_{(\mathbf{L}_{\alpha},S^R_{\alpha})}(\dd \mathbf{x},\dd t)\stackrel{\mathrm{def}}{=} \int_{\mathbf{D}}\delta_{(t\mathbf{u})}(\dd \mathbf{x})\nu_R(\dd t)\mathbf{\Lambda}(\dd \mathbf{u})
$$
Therefore condition (a) of Th. 3.2.2 \cite{Meerschaert1} is fulfilled. .$\Box$

\section*{Appendix C. Proof of Theorem \ref{th:fconv}}

The functional convergence of wait-first and jump-first L\'evy walk with rests in Thereom \ref{th:fconv} results from the functional convergence of joint waiting times and jumps obtained in Lemma \ref{lemma:joinfconv}
To prove the functional convergence given by (\ref{jointfconv1}) and (\ref{jointfconv3}) the parts (i) and (iii) of Th. \ref{th:fconv} let us take the array of random vectors
$\{(n^{-1/\alpha} \mathbf{J}_i, n^{-1/\alpha} R_i, n^{-1/\alpha} T_i)\}$ and define two auxiliary sequences of wait-first and jump-first LW:
$$
\mathbf{U}^R_n(t) = n^{-\gamma/\alpha} \left(\sum_{i=1}^{N^R_n(t)} \mathbf{J}_i-\mathbf{b}_{N(n^{1/\alpha}t)}  \right), \qquad \mathbf{{O}}^R_n(t) = n^{-1/\alpha} \left(\sum_{i=1}^{N^R_n(t)+1} \mathbf{J}_i -\mathbf{b}_{N(n^{1/\alpha}t)} \right),
$$
where
$$
N^R_n(t) = \max\left\{k \geq 0: \sum_{i=1}^k n^{-1/\alpha}(T_i+R_i) \leq t\right\}.
$$
Since Lemma \ref{lemma:joinfconv} holds and processes $S_\alpha(t)$ and $S^R_\alpha(t)$ have strictly increasing realizations from Th. 3.6 \cite{HenryStraka} it follows that the following convergences hold:
$$
\mathbf{U}^R_n(t) \fconv \left( \mathbf{L}^{-}_{\alpha}\left( S^{-1}_{\alpha}(t)\right)\right)^+, \quad \mathbf{O}^R_n(t) \fconv \mathbf{L}^{}_{\alpha}\left( (S^R)^{-1}_{\alpha}(t)\right),
$$
as $n\to\infty$.\\
Finally observe that  $N^R_n(t)= N^R(n^{1/\alpha}t)$,
therefore $\mathbf{U}^R_n(t) = n^{-1/\alpha} \mathbf{U}^R(n^{-1/\alpha}t)$ and $\mathbf{O}^R_n(t) = n^{-1/\alpha} \mathbf{O}(n^{-1/\alpha}t)$, which completes the proof of part (i) of Th. \ref{th:fconv}. In analogous way we prove the functional convergence (\ref{jointfconv2})  in part (ii) of Th. \ref{th:fconv}.
The form of the L\'evy triplets of the limiting processes follows directly from Lemma \ref{lemma:joinfconv}.$\Box$


\vskip2pc
\bibliographystyle{elsarticle-harv}
\bibliography{sample} 
\end{document}